\documentclass[a4paper,11pt,reqno]{amsart}
\usepackage{amsfonts,amsmath,amssymb,amsthm}
\usepackage{graphicx}
\usepackage{fixmath}
\usepackage{hyperref}
\usepackage{tikz}
\usepackage{a4wide}
\usepackage{cite}

\newtheorem{theo}{Theorem}
\newtheorem{lemma}{Lemma}
\newtheorem{cor}{Corollary}
\newtheorem{conj}{Conjecture}

\newcommand{\Prob}{\mathbb{P}}
\newcommand{\Ex}{\mathbb{E}}

\newcommand{\spp}{P}
\newcommand{\prs}{\mathcal{P}}
\newcommand{\st}{s}
\newcommand{\tot}{T}

\newcommand{\Sur}{\operatorname{Sur}}

\begin{document}

\thispagestyle{plain}

\title{On the probability that a random subtree is spanning}

\author{Stephan Wagner}
\thanks{The author gratefully acknowledges financial support from a Swedish Foundations' Starting
Grant from the Ragnar S\"oderberg Foundation.}
\address{Stephan Wagner\\
        Department of Mathematical Sciences\\
        Stellenbosch University\\
        Private Bag X1\\
        Matieland 7602\\
        South Africa
\and
Department of Mathematics \\
Uppsala Universitet \\
Box 480 \\
751 06 Uppsala \\
Sweden
}
\email{swagner@sun.ac.za,stephan.wagner@math.uu.se}

\subjclass[2010]{05C05; 05C80}
\keywords{subtrees, spanning trees, random graphs, convergence in probability}

\begin{abstract}
We consider the quantity $\spp(G)$ associated with a graph $G$ that is defined as the probability that a randomly chosen subtree of $G$ is spanning. Motivated by conjectures due to Chin, Gordon, MacPhee and Vincent on the behaviour of this graph invariant depending on the edge density, we establish first that $\spp(G)$ is bounded below by a positive constant provided that the minimum degree is bounded below by a linear function in the number of vertices. Thereafter, the focus is shifted to the classical Erd\H{o}s-R\'enyi random graph model $G(n,p)$. It is shown that $\spp(G)$ converges in probability to $e^{-1/(ep_{\infty})}$ if $p \to p_{\infty} > 0$ and to $0$ if $p \to 0$.
\end{abstract}

\maketitle

\section{Introduction}
\label{sec:intro}

The study of the distribution of subtrees of a tree goes back to Jamison's work~\cite{Jamison1983average,Jamison1984Monotonicity} in the 1980s. Recently, there has been renewed interest especially in the average order of subtrees of a tree, and several problems left open by Jamison have been resolved \cite{Wagner2016local,Haslegrave2014Extremal,Vince2010average,Mol2019Maximizing}. The recent paper~\cite{Chin2018Subtrees} by Chin, Gordon, MacPhee and Vincent  initiated the investigation of several graph invariants based on subtrees of an arbitrary graph, not necessarily a tree. 
One of these quantities is the probability $\spp(G)$ that a randomly chosen subtree of a graph $G$ is spanning. Clearly, this probability is positive if and only if the graph is connected.

More formally, let $\st_k(G)$ be the number of subtrees of a graph $G$ of order $k$---in other words, the number of (not necessarily induced) $k$-vertex subgraphs of $G$ that are trees. In particular, if $n$ is the number of vertices of $G$, then $\st_n(G)$ is the number of spanning trees. We write $\tot(G) = \st_1(G) + \st_2(G) + \cdots + \st_n(G)$ for the total number of subtrees of $G$. Note that
$$\spp(G) = \frac{\st_n(G)}{\tot(G)} = \frac{\st_n(G)}{\sum_{k=1}^n \st_k(G)}$$
for every graph $G$ with $n$ vertices.

\medskip

Chin et al.~proved in~\cite{Chin2018Subtrees} (and earlier in~\cite{Chin2015Pick}) that $P(K_n)$, i.e., the probability that a random subtree of the complete graph $K_n$ is spanning, tends to $e^{-1/e}$ as $n \to \infty$. For complete bipartite graphs, they found that $\lim_{n \to \infty} P(K_{n,n}) = e^{-2/e}$. Looking at very sparse graphs, on the other hand, one notices that the probability $\spp(G)$ is very small, e.g. when the graph $G$ is itself a tree. In this case, it clearly goes to $0$ as the number of vertices of $G$ goes to $\infty$.

\medskip

This leads to the natural conjecture that $\spp(G)$ is ``small'' (in a certain sense) if the graph $G$ is sparse, and ``large'' if $G$ is dense. However, it is quite difficult to make this statement precise. As Chin et al.~already point out themselves, there are rather dense graphs for which the probability is still quite small: for example, let $G$ be a complete graph $K_n$ with a path of length $\omega(n)$ attached to one of the vertices. As long as $\omega(n)$ goes to $\infty$ (arbitrarily slowly with $n$), we have $\spp(G) \leq \frac{1}{\omega(n)} \to 0$, while the edge density of $G$ is close to $1$.

Chin et al.~\cite{Chin2018Subtrees} suggested the following conjecture, which avoids such pathological cases.

\begin{conj}[Chin et al.~\cite{Chin2018Subtrees}]\label{conj:dense}
Suppose that $G_n$ is a sequence of connected graphs where $G_n$ has $n$ vertices and $\Omega(n^2)$ edges, and that $G_n$ is edge-transitive. Then $\liminf_{n \to \infty} \spp(G_n) > 0$.
\end{conj}

In this paper, we prove a modified version of this conjecture: as it turns out, the assumption that the minimum degree is linear in the number of vertices guarantees $\spp(G)$ to be bounded below by a positive constant, without any additional symmetry condition. Specifically, we have the following theorem.

\begin{theo}\label{thm:lower}
Let $\alpha$ be a fixed positive constant. There exists a constant $c = c(\alpha)$ with the following property:
for every connected graph $G$ with $n$ vertices and minimum degree $\delta = \delta(G) \geq \alpha n$, we have
$$\spp(G) \geq e^{-1/\alpha - c/n}.$$
\end{theo}

Since edge-transitive graphs are either regular or bipartite and biregular, Conjecture~\ref{conj:dense} is actually implied by this theorem. At the same time, Chin et al.~conjectured in~\cite{Chin2018Subtrees} that an edge density that goes to $0$ will always force $\spp(G)$ to go to $0$ as well. More precisely, they made the following conjecture.

\begin{conj}[Chin et al.~\cite{Chin2018Subtrees}]\label{conj:sparse}
Suppose that $G_n$ is a sequence of connected graphs where $G_n$ has $n$ vertices and $O(n^{\alpha})$ edges, where $\alpha < 2$ is fixed. Then $\lim_{n \to \infty} \spp(G_n) = 0$.
\end{conj}

Unfortunately, this conjecture is false: consider a graph $G$ consisting of a path of length $n$ with a complete graph of order $\lfloor \sqrt{n} \rfloor$ attached to each end. The number of edges of such a graph is only linear in the number of vertices. However, it is not difficult to see that almost all subtrees have to contain vertices of both complete graphs, thus also the entire connecting path. Then it follows essentially from the aforementioned result on complete graphs that $\spp(G) \to e^{-2/e}$ as $n \to \infty$. The sizes of the two complete graphs can even be reduced further to $\lfloor n^{\epsilon} \rfloor$ vertices to obtain a graph with only $n + O(n^{2\epsilon})$ edges and $n + O(n^{\epsilon})$ vertices for every $\epsilon > 0$, still with the same limit.

\medskip

In this paper, we will show, however, that Conjecture~\ref{conj:sparse} is true for sparse \emph{random} graphs (Erd\H{o}s-R\'enyi graphs $G(n,p)$, where $p \to 0$), while on the other hand Conjecture~\ref{conj:dense} is true for dense random graphs (Erd\H{o}s-R\'enyi graphs $G(n,p)$, where $p \to p_{\infty} > 0$). The precise statement reads as follows.

\begin{theo}\label{thm:gnp}
Consider the Erd\H{o}s-R\'enyi random graph $G = G(n,p)$, where the probability $p$ is allowed to depend on the number of vertices $n$. As $n \to \infty$, we have
\begin{itemize}
\item $\spp(G) \overset{p}{\to} e^{-1/(ep_{\infty})}$ if $p \to p_{\infty} > 0$,
\item $\spp(G) \overset{p}{\to} 0$ if $p \to 0$.
\end{itemize}
\end{theo}

As a side result of our considerations, we also obtain information on the distribution of the number of subtrees in dense Erd\H{o}s-R\'enyi graphs, i.e., in the case where $p$ is constant or converges to a positive constant $p_{\infty}$. This is based on the following distributional result due to Janson~\cite{Janson1994numbers} for the number of spanning trees in $G(n,p)$, which we will also use for the purposes of our proof.

\begin{theo}[Janson~\cite{Janson1994numbers}]\label{thm:janson}
Assume that $p \to p_{\infty} \in [0,1)$ and that $\liminf_{n \to \infty} \sqrt{n}p > 0$. Then we have, as $n \to \infty$,
$$p^{1/2} \Big( \log \st_n(G(n,p)) - \log (n^{n-2}p^{n-1}) + \frac{1-p}{p} \Big) \overset{d}{\to} N(0,2(1-p_{\infty})).$$
On the other hand, if $p \to 1$ and $n^2(1-p) \to \infty$, then we have, as $n \to \infty$,
$$\frac{1}{\sqrt{2(1-p)}} \Big( \frac{\st_n(G(n,p))}{\mathbb{E} (\st_n(G(n,p)) )} - 1 \Big) \overset{d}{\to} N(0,1).$$
\end{theo}

Theorem~\ref{thm:lower} will be proven in the following section. The double-counting argument that is used will be important for our treatment of random graphs as well. The dense case of Theorem~\ref{thm:gnp} will be treated in Section~\ref{sec:dense}, the sparse case in Section~\ref{sec:sparse}. In the sparse case, we even obtain a somewhat stronger result than stated, which puts an explicit upper bound on $\spp(G(n,p))$ that holds with high probability. Some corollaries of our main results will be discussed in the concluding section.

\section{Graphs with large minimum degree}\label{sec:mindeg}

This section is concerned with the proof of our first main result, Theorem~\ref{thm:lower}. The proof technique, however, will also be important for the proofs of our other results. As it turns out, the condition that the minimum degree is ``large'', i.e., linear in the number of vertices, already suffices to obtain a lower bound on the probability $\spp(G)$.

\begin{proof}[Proof of Theorem~\ref{thm:lower}]
We use a double-counting argument to prove the statement of the theorem: for every positive integer $k$, we count the number of pairs $(S,T)$ consisting of a subtree $S$ of $G$ with $n-k$ vertices and a spanning tree $T$ of $G$ such that $S$ is a subtree of $T$. Let this number be denoted by $\prs_k(G)$.
\begin{itemize}
\item On the one hand, for every spanning tree $T$ of $G$, the number of possible subtrees $S$ to form a pair with $T$ that satisfies the required conditions is clearly at most $\binom{n}{k}$ (which is the number of ways to choose $n-k$ vertices). Therefore,
$$\prs_k(G) \leq \binom{n}{k} \st_n(G) \leq \frac{n^k}{k!} \st_n(G).$$
\item On the other hand, we can construct a feasible pair $(S,T)$ starting from the subtree $S$. For each of the remaining $k$ vertices, we choose one of the edges that connects it with one of the vertices of $S$. Adding all these edges to $S$, we obtain a tree $T$; note that all the vertices of $T$ that do not belong to $S$ are leaves by this construction. There are at least $\delta - k$ edges to choose for each of the $k$ vertices, so it follows that
$$\prs_k(G) \geq (\delta - k)^k \st_{n-k}(G)$$
for every $k \leq \delta$.
\end{itemize}
Combining the two inequalities, we find that
$$(\delta - k)^k \st_{n-k}(G) \leq \prs_k(G) \leq \frac{n^k}{k!} \st_n(G)$$
and thus
\begin{equation}\label{eq:firstineq}
\frac{\st_{n-k}(G)}{\st_n(G)} \leq \frac{n^k}{(\delta - k)^k k!} \leq \frac{n^k}{(\alpha n - k)^k k!} = \frac{1}{\alpha^k k!} \Big( 1 - \frac{k}{\alpha n} \Big)^{-k}
\end{equation}
for every $k < \alpha n \leq \delta$.

We need one more simple inequality for our purposes: clearly, every subtree of $G$ with at most $r$ vertices is contained in a subtree with exactly $r$ vertices, since $G$ is connected (which allows us to add vertices one by one until we reach a tree with exactly $r$ vertices). On the other hand, an $r$-vertex subtree contains at most $2^r$ smaller subtrees, hence
\begin{equation}\label{eq:secondineq}
\st_1(G) + \st_2(G) + \cdots + \st_r(G) \leq 2^r \st_r(G).
\end{equation}

Now we split the sum
$$\frac{1}{P(G)} = \frac{\st_1(G) + \st_2(G) + \cdots + \st_n(G)}{\st_n(G)} = \sum_{k=0}^{n-1} \frac{\st_{n-k}(G)}{\st_n(G)}$$
into three parts that are estimated separately. It clearly suffices to prove the desired inequality for sufficiently large $n$.
\begin{itemize}
\item For $k \leq (\alpha n)^{1/6}$ (the specific choice of exponent is irrelevant, provided it is less than $\frac12$; we choose $\frac16$ for later use), we use~\eqref{eq:firstineq}. Observe that
$$\Big( 1 - \frac{k}{\alpha n} \Big)^{-k} = \exp \Big(\negthickspace - k \log \Big( 1 - \frac{k}{\alpha n} \Big) \Big) = \exp \Big( \frac{k^2}{\alpha n} + O \Big( \frac{k^3}{n^2} \Big) \Big) = 1 + O \Big( \frac{k^2}{n} \Big)$$
holds under our assumption on $k$. It follows that
$$\sum_{0 \leq k \leq (\alpha n)^{1/6}} \frac{\st_{n-k}(G)}{\st_n(G)} \leq \sum_{0 \leq k \leq (\alpha n)^{1/6}} \frac{1}{\alpha^k k!} \Big( 1 + \frac{\kappa k^2}{n} \Big)$$
for some constant $\kappa > 0$ (that depends on $\alpha$), and we find that
$$\sum_{0 \leq k \leq (\alpha n)^{1/6}} \frac{\st_{n-k}(G)}{\st_n(G)} \leq \sum_{k \geq 0} \frac{1}{\alpha^k k!} \Big( 1 + \frac{\kappa k^2}{n} \Big) = e^{1/\alpha} \Big( 1 + \frac{\kappa(1+\alpha)}{\alpha^2 n} \Big).$$
\item Next, we consider the case that $(\alpha n)^{1/6} < k \leq \frac{\alpha n}{2}$. In this case, the estimate~\eqref{eq:firstineq} yields
$$\frac{\st_{n-k}(G)}{\st_n(G)} \leq \frac{2^k}{\alpha^k k!}.$$
For large enough $n$, this gives us
\begin{align*}
\sum_{(\alpha n)^{1/6} < k \leq \alpha n/2} \frac{\st_{n-k}(G)}{\st_n(G)} &\leq \sum_{k > (\alpha n)^{1/6}} \frac{2^k}{\alpha^k k!} \leq \frac{2^{\lceil (\alpha n)^{1/6} \rceil}}{\alpha^{\lceil (\alpha n)^{1/6} \rceil} (\lceil (\alpha n)^{1/6} \rceil)!} \sum_{\ell \geq 0} \frac{2^\ell}{\alpha^\ell (\alpha n)^{\ell/6}} \\
&= \frac{2^{\lceil (\alpha n)^{1/6} \rceil}}{\alpha^{\lceil (\alpha n)^{1/6} \rceil} (\lceil (\alpha n)^{1/6} \rceil)!} \Big( 1 - \frac{2}{\alpha^{7/6} n^{1/6}} \Big)^{-1},
\end{align*}
which goes to $0$ faster than any power of $n$ thanks to the factorial in the denominator.
\item Finally, we consider the case that $k > \frac{\alpha n}{2}$. Here, we combine~\eqref{eq:firstineq} and~\eqref{eq:secondineq} to obtain
$$\sum_{k > \alpha n/2} \frac{\st_{n-k}(G)}{\st_n(G)} \leq \frac{\st_{n-\lfloor \alpha n/2 \rfloor}(G)2^{n-\lfloor \alpha n/2 \rfloor}}{\st_n(G)} \leq \frac{2^{\lfloor \alpha n/2 \rfloor}2^{n-\lfloor \alpha n/2 \rfloor}}{\alpha^{\lfloor \alpha n/2 \rfloor} (\lfloor \alpha n/2 \rfloor)!} = \frac{2^n}{\alpha^{\lfloor \alpha n/2 \rfloor} (\lfloor \alpha n/2 \rfloor)!},$$
which also goes to $0$ faster than any power of $n$.
\end{itemize}
Combining the three cases, we arrive at
$$\frac{1}{\spp(G)} = \sum_{k=0}^{n-1} \frac{\st_{n-k}(G)}{\st_n(G)} \leq e^{1/\alpha} \Big( 1 + O \Big( \frac{1}{n} \Big) \Big) = e^{1/\alpha + O(1/n)},$$
and the statement of the theorem follows.
\end{proof}

\section{Dense random graphs}\label{sec:dense}

Next we consider random Erd\H{o}s-R\'enyi graphs $G(n,p)$ where $p$ is constant or converges to a positive constant $p_{\infty}$. The argument of the previous section can be modified to obtain more than just a lower bound in this case. We can exploit the fact that random graphs $G(n,p)$, where $p$ is of constant order, are ``almost regular'' to prove that $\spp(G(n,p))$ converges in probability to a constant, giving us the first half of Theorem~\ref{thm:gnp}. Let us first repeat the formal statement.

\begin{theo}\label{thm:dense}
Suppose that $p \to p_{\infty} > 0$ as $n \to \infty$. We have
$$\spp(G(n,p)) \overset{p}{\to} e^{-1/(e p_{\infty})}$$
as $n \to \infty$.
\end{theo}

For the proof, we will use a modification of inequality~\eqref{eq:firstineq}. It is based on the same double-counting strategy, but the derivation is somewhat more involved. It will therefore be presented as a sequence of lemmas. Let us first mention two classical tools that will be used repeatedly in the following: first, \emph{Markov's inequality} on nonnegative random variables, see for instance~\cite[Chapter 3, Theorem 1.1]{Gut2013Probability} or \cite[Eq.~(1.3)]{Janson2000Random}.

\begin{lemma}[Markov's inequality]
For every nonnegative random variable $X$ and every positive real number $a$, we have $\Prob(X \geq a) \leq \frac{\Ex(X)}{a}$.
\end{lemma}

There are various variants of the so-called \emph{Chernoff bounds} for bounding the tails of sums of independent random variables, in particular the tails of the binomial distribution. We will use the following version~\cite[Theorem 2.1]{Janson2000Random}:

\begin{lemma}[Chernoff bounds]\label{lem:chernoff}
Let $X$ be a random variable that follows a $\operatorname{Bin}(n,p)$ distribution, and write $\mu = \Ex(X) = np$. For every $t > 0$, we have
$$\Prob \big(X \leq \mu - t\big) \leq \exp \Big( \negthickspace -\frac{t^2}{2 \mu} \Big)$$
and
$$\Prob \big(X \geq \mu + t\big) \leq \exp \Big( \negthickspace -\frac{t^2}{2(\mu+t/3)} \Big).$$
\end{lemma}

Next, we collect some properties of $G(n,p)$ that hold with high probability, i.e., with probability tending to $1$.

\begin{lemma}\label{lem:four_statements}
Suppose that $p \to p_{\infty} > 0$ as $n \to \infty$. Each of the following events occurs with high probability as $n \to \infty$:
\begin{enumerate}
\item The maximum degree $\Delta(G(n,p))$ is bounded above by $pn + n^{2/3}$.
\item The minimum degree $\delta(G(n,p))$ is bounded below by $pn - n^{2/3}$.
\item The number of spanning trees $\st_n(G(n,p))$ is bounded below by $n^{n-2}p^{n-1} e^{-n^{1/6}}$.
\item The number of spanning trees whose number of leaves does not lie in the interval $[\frac{n}{e} - n^{2/3},\frac{n}{e} + n^{2/3}]$ is bounded above by $n^{n-2}p^{n-1}e^{-n^{1/4}}$.
\end{enumerate}
\end{lemma}

\begin{proof}
The first two statements are classical: note that the degree of a fixed vertex is the sum of $n-1$ independent $0$-$1$-random variables, thus follows a binomial distribution. Using the second inequality of Lemma~\ref{lem:chernoff}, we find that the probability for any fixed vertex to have degree greater than $pn + n^{2/3}$ is at most $e^{-cn^{1/3}}$ for some constant $c > 0$ if $n$ is large enough, so the expected number of vertices whose degree is greater than $pn + n^{2/3}$ is at most $n e^{-cn^{1/3}}$. Combined with Markov's inequality, this shows that the probability that at least one of the vertices has degree greater than $pn + n^{2/3}$ must tend to $0$. The same argument applies to the minimum degree.

\medskip

For the third statement, we use Janson's results (see Theorem~\ref{thm:janson}) on the distribution of the number of spanning trees: our statement can be rewritten as
$$\log \st_n(G(n,p)) - \log (n^{n-2}p^{n-1}) \geq -n^{1/6},$$
and Janson's distributional results clearly imply that this holds with probability tending to $1$ as long as $(1-p)n^2$ still goes to $\infty$.  If $p$ goes even faster to $1$, then one can e.g.~apply Janson's theorem with $p = 1-n^{-3/2}$ and use the fact that the number of spanning trees increases when edges are added.

\medskip

Let us finally deal with the last statement: we show that the expected number of such spanning trees is substantially smaller than the given bound, so that we can apply Markov's inequality. To this end, we need a bound on the number of trees on a set of $n$ vertices whose number of leaves does not lie in the interval $[\frac{n}{e} - n^{2/3},\frac{n}{e} + n^{2/3}]$. It is known that the number of leaves in a random labelled tree with $n$ vertices (equivalently, a random spanning tree of the complete graph $K_n$) satisfies a central limit theorem with mean $\sim \frac{n}{e}$ and variance $\sim \frac{(e-2)n}{e^2}$. Moreover, an exponential tail bound holds: for some constants $\kappa,\lambda > 0$, the number of trees with fewer than $\frac{n}{e} - t$ or more than $\frac{n}{e} + t$ leaves is at most $n^{n-2}e^{-\kappa t^2/n}$, provided that $t \leq \lambda n$. One way to prove this is to note that the bivariate generating function for labelled trees (where the first variable is associated with the number of vertices and the second with the number of leaves) satisfies the conditions of \cite[Theorems 2.21--2.23]{Drmota2009Random} (cf. \cite[Theorem 3.13]{Drmota2009Random}). Specifically, inequality (2.26) of \cite[Theorem 2.22]{Drmota2009Random} gives the desired tail bound.

Each tree has probability $p^{n-1}$ of occurring as spanning tree in $G(n,p)$, so the expected number of spanning trees satisfying the stated condition is at most $n^{n-2}p^{n-1}e^{-\kappa n^{1/3}}$. Now another standard application of Markov's inequality completes the proof.
\end{proof}

Throughout the rest of this section, we assume that $p \to p_{\infty} > 0$ as $n \to \infty$. By Lemma~\ref{lem:four_statements}, we can assume in the following that the four statements of the lemma are all satisfied. This turns out to be sufficient for our purposes. Let us recall the notation $\prs_k(G)$ from the previous section, which stands for the number of pairs $(S,T)$ of a spanning tree $T$ of $G$ and a subtree $S$ of $T$ such that $|S| = |T| - k$.

\begin{lemma}\label{lem:Pk1}
Suppose that a graph $G$ with $n$ vertices satisfies the first two conditions of Lemma~\ref{lem:four_statements}. Define $\prs_k(G)$ as in the proof of Theorem~\ref{thm:lower}. We have, for every $k \leq n^{1/6}$,
$$\prs_k(G) = \st_{n-k}(G) p^k n^k \big(1 + O(n^{-1/6})\big),$$
where the $O$-constant does not depend on $k$.
\end{lemma}

\begin{proof}
We determine the number $\prs_k(G)$ of pairs $(S,T)$ consisting of a spanning tree $T$ and a subtree $S$ with $n-k$ vertices contained in $T$ by first fixing $S$ (for which we have $\st_{n-k}(G)$ possibilities) and counting the number of ways to extend $S$ to a spanning tree $T$.

Clearly, each of the $k$ vertices that do not belong to $S$ can be connected to one of the vertices of $S$ by one of at least $\delta(G) - k \geq pn - n^{2/3} - k$ edges. This means that the number of possible extensions to a tree $T$ is at least
$$(pn - n^{2/3} - k)^k = p^k n^k \big(1 - O(n^{-1/3})\big)^k = p^k n^k \big( 1 - O(kn^{-1/3}) \big) = p^k n^k \big( 1 - O(n^{-1/6}) \big).$$
Now we prove a matching upper bound. The graph induced by $T$ on the vertices that do not belong to $S$ must be a forest and thus consist of $r$ edges for some $r < k$. The number of components is then exactly $k-r$, and each of them must be connected to $S$ by precisely one edge.

\medskip

There are clearly at most $\binom{k(k-1)/2}{r}$ choices for the induced forest (since it consists of a subset of the $\binom{k}{2} = \frac{k(k-1)}{2}$ edges). If the component sizes of the forest are $a_1,a_2,\ldots,a_{k-r}$, then there are $a_1a_2\cdots a_{k-r}$ ways to choose the vertices that have a neighbour in $S$, and at most $\Delta(G)^{k-r}$ choices for the vertices in $S$ that they are connected to. By the inequality between the arithmetic and the geometric mean, we have
$$a_1a_2\cdots a_{k-r} \leq \Big( \frac{a_1+a_2+\cdots+a_{k-r}}{k-r} \Big)^{k-r} = \Big( \frac{k}{k-r} \Big)^{k-r} = \Big(1 - \frac{r}{k} \Big)^{r(1-k/r)}.$$
Noting that the function $x \mapsto (1-x)^{1-1/x}$ is bounded above by $e$ for $x \in (0,1]$, we can conclude that
$$a_1a_2\cdots a_{k-r} \leq e^r.$$
Thus the total number of ways to extend $S$ to a tree $T$ is at most
\begin{align*}
\sum_{r=0}^{k-1} \binom{k(k-1)/2}{r} e^r \Delta(G)^{k-r} &\leq \Delta(G)^k \sum_{r=0}^{k(k-1)/2} \binom{k(k-1)/2}{r} \Big(\frac{e}{\Delta(G)} \Big)^r \\
&= \Delta(G)^k\Big(1 + \frac{e}{\Delta(G)}\Big)^{k(k-1)/2}.
\end{align*}
Since we are assuming that $pn - n^{2/3} \leq \delta(G) \leq \Delta(G) \leq pn + n^{2/3}$ and $k \leq n^{1/6}$, we have
$$\Delta(G)^k = (pn)^k \big(1 + O(n^{-1/3})\big)^k = p^k n^k \big( 1 + O(kn^{-1/3}) \big) = p^k n^k \big( 1 + O(n^{-1/6}) \big)$$
as well as
$$\Big(1 + \frac{e}{\Delta(G)}\Big)^{k(k-1)/2} \leq \exp \Big( \frac{ek(k-1)}{2\Delta(G)} \Big) = 1 + O(k^2 n^{-1}) = 1 + O(n^{-2/3}),$$
so the desired upper bound follows as well.
\end{proof}

For our next step, we need bounds on the number of subtrees of a tree.

\begin{lemma}\label{lem:est_by_leaves}
Let $T$ be a tree with $n$ vertices, of which $\ell$ are leaves. The number of subtrees of $T$ with $n-k$ vertices satisfies
$$\binom{\ell}{k} \leq \st_{n-k}(T) \leq \binom{\ell+k-1}{k}.$$
\end{lemma}

\begin{proof}
The lower bound is trivial: removing any subset of $k$ leaves from $T$, we obtain a subtree with $n-k$ vertices. For the upper bound, we associate every subtree $S$ of $T$ with an $\ell$-tuple $(a_1,a_2,\ldots,a_\ell)$ of nonnegative integers. To this end, let us denote the leaves by $v_1,v_2,\ldots,v_{\ell}$. For $0 \leq j \leq \ell$, we define $S_j$ as the smallest subtree of $T$ that contains $S$ as well as $v_1,v_2,\ldots,v_j$. In particular, $S_0 = S$ and $S_{\ell} = T$. 

\medskip

It is easy to see that $S_j$ is obtained from $S_{j-1}$ by adding the path from $v_j$ to the vertex of $S_{j-1}$ nearest to it (this path might be empty if $v_j$ is already contained in $S_{j-1}$). Now set $a_j = |S_j| - |S_{j-1}|$ for $1 \leq j \leq \ell$ to obtain the $\ell$-tuple $(a_1,a_2,\ldots,a_{\ell})$. In other words, $a_j$ is the number of vertices on the path that is added to $S_{j-1}$ to obtain $S_j$ (including $v_j$ if it is not in $S_{j-1}$, but not the other end).

\medskip

Now we claim that the subtree $S$ can be reconstructed uniquely from $(a_1,a_2,\ldots,a_{\ell})$ (if there is any subtree associated with the specific $\ell$-tuple). We use backwards induction on $j$ to show that $S_j$ is unique for $0 \leq j \leq \ell$. This is trivial for $j = \ell$, so we focus on the induction step. Assume that $S_j$ is known for some $j > 0$, and let $w$ be the vertex nearest to $v_j$ in $S_j$ whose degree is greater than $2$. If no such vertex exists, then $S_j$ is a path with $v_j$ at one of its ends, and we let $w$ be the other end (this is only possible if $j=1$ or $j=2$).

\medskip

In the former case (the degree of $w$ is greater than $2$), consider the components of $S_j - w$: each of them needs to contain either a vertex of $S$ or one of the leaves $v_1,v_2,\ldots,v_j$, since we could otherwise remove it from $S_j$ to obtain a smaller tree that still contains all of $S$ as well as $v_1,v_2,\ldots,v_j$, contradicting the definition of $S_j$. Since there are at least three such components, at least two of them contain a vertex of $S$ or one of $v_1,v_2,\ldots,v_{j-1}$. This implies that $w$ must be a vertex of $S_{j-1}$. Note that this is also true in the latter case, where $S_j$ is a path and $w$ the other end: if $w$ was not part of $S$ or the leaf $v_1$, we could remove it from $S_j$.

\medskip

So we find that the difference between $S_{j-1}$ and $S_j$ is always a part of the path between $w$ and $v_j$ that is either empty or contains $v_j$. This means that it is uniquely determined by its number of vertices $a_j$, and we can determine $S_{j-1}$ from $S_j$ and $a_j$, completing our induction proof that $S$ is uniquely determined by the $\ell$-tuple $(a_1,a_2,\ldots,a_{\ell})$.

\medskip

Since
$$|T| - |S| = \sum_{j=1}^{\ell} \big( |S_j| - |S_{j-1}| \big) = \sum_{j=1}^{\ell} a_j,$$
we find that $\st_{n-k}(T)$ is bounded above by the number of $\ell$-tuples $(a_1,a_2,\ldots,a_{\ell})$ of nonnegative integers that satisfy
$$\sum_{j=1}^{\ell} a_j = k,$$
which is well known to be $\binom{\ell+k-1}{k}$. This completes the proof of the upper bound.
\end{proof}

\begin{lemma}\label{lem:Pk2}
Suppose that a graph $G$ with $n$ vertices satisfies the last two conditions of Lemma~\ref{lem:four_statements}. Define $\prs_k(G)$ as in the proof of Theorem~\ref{thm:lower}. We have, for every $k \leq n^{1/6}$,
$$\prs_k(G) = \st_{n}(G) \frac{n^k}{e^k k!} \big(1 + O(n^{-1/6})\big),$$
where the $O$-constant does not depend on $k$.
\end{lemma}

\begin{proof}

For every tree $T$ with a number of leaves $\ell$ between $\frac{n}{e} - n^{2/3}$ and $\frac{n}{e} + n^{2/3}$, we use the bounds of Lemma~\ref{lem:est_by_leaves}. For $k \leq n^{1/6}$, we have
$$\frac{(n/e - n^{2/3} - n^{1/6})^k}{k!}
\leq \binom{\ell}{k} \leq \st_{n-k}(T) \leq \binom{\ell+k-1}{k} \leq \frac{(n/e+n^{2/3} + n^{1/6})^k}{k!},$$
and both the upper and lower bound are of the form
$$\frac{n^k}{e^k k!} \big( 1 + O(n^{-1/3}) \big)^k = \frac{n^k}{e^k k!} \big(1 + O(k n^{-1/3}) \big) = \frac{n^k}{e^k k!} \big(1 + O(n^{-1/6}) \big).$$
Therefore, the contribution of all these trees to $\prs_k(G)$ is
\begin{align*}
\Big( \st_n(G) - O(n^{n-2}p^{n-1}e^{-n^{1/4}}) \Big) &\frac{n^k}{e^k k!} \big(1 + O(n^{-1/6}) \big) \\
&= \st_n(G) \Big( 1 - O\big(e^{n^{1/6}-n^{1/4}}\big) \Big) \frac{n^k}{e^k k!} \big(1 + O(n^{-1/6}) \big) \\
&= \st_n(G) \frac{n^k}{e^k k!} \big(1 + O(n^{-1/6}) \big).
\end{align*}
On the other hand, every tree $T$ with $n$ vertices has at most $\binom{n}{k} \leq \frac{n^k}{k!}$ subtrees with $n-k$ vertices. Therefore, the contribution of spanning trees $T$ with more than $\frac{n}{e}  + n^{2/3}$ or fewer than $\frac{n}{e} - n^{2/3}$ leaves to $\prs_k(G)$ is at most
$$n^{n-2} p^{n-1} e^{-n^{1/4}} \cdot \frac{n^k}{k!} \leq \st_n(G) e^{2n^{1/6}-n^{1/4}} \cdot \frac{n^k}{e^k k!}$$ 
by our assumptions on the number of such spanning trees. Since the exponent $2n^{1/6} - n^{1/4}$ goes to $-\infty$, this contribution is negligible, and we conclude that
$$\prs_k(G) = \st_n(G) \frac{n^k}{e^k k!} \big(1 + O(n^{-1/6})\big).$$
\end{proof}

Now we can put together all ingredients to prove Theorem~\ref{thm:dense}.

\begin{proof}[Proof of Theorem~\ref{thm:dense}]
We can assume that all four statements of Lemma~\ref{lem:four_statements} are satisfied, since each of them holds with high probability.

\medskip

For $k \leq n^{1/6}$, Lemmas~\ref{lem:Pk1} and ~\ref{lem:Pk2} yield
$$\st_{n-k}(G) p^k n^k \big(1 + O(n^{-1/6})\big) = \prs_k(G) = \st_{n}(G) \frac{n^k}{e^k k!} \big(1 + O(n^{-1/6})\big),$$
thus
$$\frac{\st_{n-k}(G)}{\st_n(G)} = \frac{1}{p^k e^k k!} \big(1 + O(n^{-1/6})\big).$$
From the proof of Theorem~\ref{thm:lower}, we know that the contribution that comes from terms with $k > n^{1/6}$ is completely negligible, given only the condition on the minimum degree from Lemma~\ref{lem:four_statements}: for every positive exponent $A$, we have
$$\sum_{k > n^{1/6}} \frac{\st_{n-k}(G)}{\st_n(G)} = O(n^{-A}).$$
Since we also have
$$\sum_{k > n^{1/6}} \frac{1}{p^k e^k k!} = O(n^{-A}),$$
we can finally conclude that
$$\frac{1}{\spp(G)} = \sum_{k=0}^{n-1} \frac{\st_{n-k}(G)}{\st_n(G)} = \sum_{k=0}^{\infty} \frac{1}{p^k e^k k!} \big(1 + O(n^{-1/6})\big) + O(n^{-A}) = e^{1/(ep)}+ O(n^{-1/6}),$$
which converges to $e^{1/(ep_{\infty})}$ by assumption.
\end{proof}

\section{Sparse random graphs}\label{sec:sparse}

In this section, we consider Erd\H{o}s-R\'enyi graphs $G(n,p)$ with $p$ tending to $0$. Based on the results of the previous section and the fact that
$$\lim_{p \to 0^+} e^{-1/(ep)} = 0,$$
one naturally expects the following theorem to hold.

\begin{theo}\label{thm:sparse}
Suppose that $p \to 0$ as $n \to \infty$. We have
$$\spp(G(n,p)) \overset{p}{\to} 0$$
as $n \to \infty$.
\end{theo}

Somewhat surprisingly, the proof becomes more complicated if $p$ is very small. This is mainly due to the fact that Janson's result on the distribution of the number of spanning trees in $G(n,p)$ (Theorem~\ref{thm:janson}) is only available when $p$ is at least of order $n^{-1/2}$. This was necessary to show that ``most'' spanning trees have sufficiently many leaves. The main technical difficulty for us will be to replace Janson's theorem by a different argument.

We will be able to prove a somewhat stronger result:

\begin{theo}\label{thm:dense_strong}
There exists an absolute constant $c > 0$ such that the inequality
$$\spp(G(n,p)) \leq e^{-c/p}$$
holds with high probability as $n \to \infty$, for any choice of $p$.
\end{theo}

Obviously, Theorem~\ref{thm:dense_strong} implies Theorem~\ref{thm:sparse}. Let us remark that the constant $c$ that we obtain through our proof is certainly not best possible. It is conceivable that it can be chosen arbitrarily close to $1/e$.

\medskip

For the proof of Theorem~\ref{thm:dense_strong}, we first require a lower bound on the total number of subtrees. The first step is a lemma on the giant component of $G(n,p)$ in the case that $pn \to \infty$.

\begin{lemma}
Suppose that $pn \to \infty$. With high probability, the graph $G(n,p)$ has a connected component that contains all but $o(1/p)$ vertices.
\end{lemma}

\begin{proof}
By \cite[Theorem 5.4]{Janson2000Random}, there exists a constant $\kappa > 0$ such that $G(n,p)$ has a connected component containing at least half of the vertices with high probability as soon as $p > \frac{\kappa}{n}$. Now we generate the edges of $G(n,p)$ in two rounds: in the first round, each edge is inserted with probability $\frac{p}{2-p}$. In the second round, each edge that is not already present is inserted with probability $\frac{p}{2}$. In this way, each edge has probability
$$\frac{p}{2-p} + \Big( 1 - \frac{p}{2-p} \Big) \cdot \frac{p}{2} = p$$
of being included, as it should be. Since $\frac{p}{2-p} > \frac{\kappa}{n}$ for large enough $n$ by our assumption on $p$, the graph after the first round contains a component of at least $\frac{n}{2}$ vertices with high probability. We assume in the following that this is the case.

For each of the remaining vertices, the probability of not getting connected to this giant component in the second round is no greater than $(1-\frac{p}2)^{n/2} \leq e^{-pn/4}$. Thus the expected number of vertices that do not belong to the giant component after the second round is bounded above by $\frac{n}{2} \cdot e^{-pn/4}$. An application of Markov's inequality now shows that with high probability, no more than $ne^{-pn/8} = o(1/p)$ of the vertices are not part of the giant component (here, we are making use of the assumption that $pn \to \infty$). This completes the proof.
\end{proof}

\begin{lemma}\label{trees_lower_bound}
There exists a constant $a  > 0$ such that
$$\tot(G(n,p)) \geq  p^n n! e^{an}$$
holds with high probability provided that $p \geq \frac{\log n}{2n}$.
\end{lemma}

\begin{proof}
Set $n_0 = \lfloor \sqrt{n/p} \rfloor$, and fix a set $A_0$ of $n_0$ vertices of $G = G(n,p)$. Since $pn_0 \to \infty$ by our assumptions on $p$, the subgraph induced by these $n_0$ vertices has a giant component that contains all but at most $o(1/p)$ vertices with high probability by the previous lemma. Let us denote the vertex set of this giant component by $A_1$, and set $n_1 = |A_1|$, so that $n_1 = n_0 - o(1/p)$. In the following, we condition on $A_1$ being a specific set of vertices, and consider $A_1$ fixed.

Next, consider an order $v_1,v_2,\ldots,v_{n-n_0}$ of the vertices that do not belong to $A_0$. Set $\epsilon = (\log n)^{-1/6}$. Inductively, we define further sets $A_2,A_3,\ldots, A_{n-n_0+1}$ of vertices as follows:
\begin{itemize}
\item If $v_i$ has at least $(1-\epsilon)p|A_i|$ neighbours among the vertices of $A_i$, set $A_{i+1} = A_i \cup \{v_i\}$.
\item Otherwise, $A_{i+1} = A_i$. In this case, we call $v_i$ a failure.
\end{itemize}
Observe that $|A_i| \geq |A_1| = n_1$. An application of the first Chernoff bound stated in Lemma~\ref{lem:chernoff} shows that a vertex is a failure with probability at most $e^{-\epsilon^2 pn_1/2} = e^{-\epsilon^2 \sqrt{p n}/2 + O(1)}$. Let us now call the order $v_1,v_2,\ldots,v_{n-n_0}$ of the vertices that do not belong to $A_0$ successful if the number of failures is no greater than $f = \lfloor n e^{-\epsilon^2 \sqrt{p n}/4} \rfloor$. By the Markov inequality, any fixed order of vertices is successful with high probability.

\medskip

The set of vertices $A_{n-n_0+1}$ has at least $n_1 + n - n_0 - f$ vertices for a successful order. We now construct subtrees of $G$ that are spanning trees of this set in the following way:
\begin{itemize}
\item Start with any spanning tree $T_1$ of $A_1$. This is possible since we are assuming that the graph induced by $A_1$ is connected.
\item Add all vertices that belong to $A_{n-n_0+1}$ (i.e., all vertices that are not failures) to the tree according to the fixed order of vertices by attaching each of the vertices $v_i$ to one of its neighbours in $|A_i|$. For each vertex $v_i$ that is added in this way, we have at least $(1-\epsilon)p|A_i|$ possible ways to do so.
\end{itemize}

This procedure gives us at least
$$\prod_{j=1}^{n-n_0-f} \big( (1-\epsilon)p(n_1+j-1)\big)$$
different trees. We can estimate this quantity as follows:
\begin{align*}
\prod_{j=1}^{n-n_0-f} \big( (1-\epsilon)p(n_1+j-1)\big) &= (1-\epsilon)^{n-n_0-f}p^{n-n_0-f} \frac{(n+n_1-n_0-f-1)!}{(n_1-1)!} \\
&\geq (1-\epsilon)^n p^{n-n_0} \frac{n!}{n_1!} n^{n_1-n_0-f-1} \\
&\geq (1-\epsilon)^n p^{n-n_0} \frac{n!}{n_0^{n_0}} n^{n_1-n_0-f-1} \\
&= p^n n! (1-\epsilon)^n (pn_0)^{-n_0} n^{n_1-n_0-f-1} \\
&= p^n n! \exp \Big( n \log(1-\epsilon) - n_0 \log(pn_0) + (n_1-n_0-f-1) \log n\Big).
\end{align*}
Observe that
\begin{itemize}
\item $n \log(1-\epsilon) = o(n)$ by our choice of $\epsilon$,
\item $n_0 \log (pn_0) \leq \sqrt{n/p} \log \sqrt{np} = n \frac{\log (np)}{2\sqrt{np}} = o(n)$ since $np \to \infty$ by our assumptions on $p$,
\item $(n_0-n_1)\log n = o(1/p \cdot \log n) = o(n)$ by definition of $n_1$,
\item $(f+1)\log n = o(n)$ by definition of $f$ (note that $e^{-\epsilon^2 \sqrt{pn}/4}$ goes faster to $0$ than any power of $\log n$).
\end{itemize}
In conclusion, the number of different trees we obtain is $p^n n! e^{o(n)}$ for every successful vertex order (and choice of spanning tree $T_1$ on $A_1$).

For the argument that follows, we need to show that most of these trees do not have too many leaves. To this end, we first determine a bound on the number of ways to successively attach at least $n - n_0 - f$ of the vertices $v_1,v_2,\ldots$ in such a way that at least $\frac{5n}{6}$ of them are leaves. Suppose that $n-n_0-j$ vertices (where $j \leq f$) of the vertices have been chosen to become part of the tree (which can be done in $\binom{n-n_0}{j}$ ways), and that we are adding them one by one, starting from $T_1$, to obtain a tree (not necessarily a subtree of $G$). Clearly, the number of ways to do so is
$$n_1(n_1+1) \cdots (n_1+n-n_0-j-1).$$
Now we estimate the number of ways to do this in such a way that the number of leaves at the end is at least $\frac{5n}{6}$. In our procedure, if a vertex ever becomes a non-leaf, it stays a non-leaf. Thus the number of non-leaves never exceeds $\frac{n}{6}$. Consequently, once the number of vertices in our tree reaches $\lceil\frac{n}{3} \rceil$, at least half of the vertices are leaves. Of the remaining $N = n - n_0 - \lceil\frac{n}{3} \rceil - j = \frac{2n}{3} - o(n)$ vertices, at most $\frac{n}{6}$ can be attached to a leaf. If we were to attach all these remaining vertices randomly, then the probability of attaching to a leaf would always be at least $\frac12$, and the probability that we attach to a leaf at most $\frac{n}{6}$ times would be at most
$$\sum_{i \leq n/6} \binom{N}{i} 2^{-N} \leq e^{-\kappa n}$$
for some $\kappa > 0$, if $n$ is sufficiently large (by the Chernoff bound of Lemma~\ref{lem:chernoff}). This means that the number of ways to obtain a tree with $n_1+n-n_0-j$ vertices for which at least $\frac{5n}{6}$ of the vertices outside of $A_1$ are leaves is at most
$$n_1(n_1+1) \cdots (n_1+n-n_0-j-1) e^{-\kappa n} = \frac{(n_1+n-n_0-j-1)!}{(n_1-1)!} e^{-\kappa n}.$$
Each of these trees has probability $p^{n-n_0-j}$ to be a subtree of $G$ (given that $T_1$ is). Therefore, the expected number of such trees with at least $\frac{5n}{6}$ leaves is at most
\begin{equation}\label{eq:exnum_trees}
\sum_{j \leq f} \binom{n-n_0}{j} p^{n-n_0-j} \frac{(n_1+n-n_0-j-1)!}{(n_1-1)!} e^{-\kappa n}.
\end{equation}
For sufficiently large $n$, we have $n_1+n-n_0-j \geq n_1+n-n_0-f \geq \frac{1}{p}$ for all $j \leq f$, thus
$$p^{n-n_0-j} \frac{(n_1+n-n_0-j-1)!}{(n_1-1)!} \leq p^{n-n_1} \frac{(n-1)!}{(n_1-1)!} \leq \frac{p^nn!}{n_1!p^{n_1}},$$
so the expression in~\eqref{eq:exnum_trees} is less than or equal to
$$\sum_{j \leq f} \binom{n-n_0}{j} \frac{p^nn!}{n_1!p^{n_1}} e^{-\kappa n}.$$
Since $f = o(n)$, we have
$$\sum_{j \leq f} \binom{n-n_0}{j} = e^{o(n)},$$
and another simple calculation shows that $n_1!p^{n_1} = e^{o(n)}$. Therefore, the expected number of trees with at least $\frac{5n}{6}$ leaves outside of $A_1$ is bounded above by
$$p^n n! e^{-\kappa n + o(n)}.$$
By the Markov inequality, there are therefore at most $p^n n! e^{-\kappa n/2}$ such trees with high probability (all of this still for a fixed vertex order).

Since we get $p^n n! e^{o(n)}$ trees from every successful vertex order, it follows that a fixed vertex order generates, with high probability, $p^n n! e^{o(n)}$ trees with the property that at most $\frac{5n}{6}$ of the vertices that do not belong to $A_1$ are leaves.

Applying the Markov inequality one more time, we see that with high probability, this statement holds for at least half of the $(n-n_0)!$ possible vertex orders.
This gives us $p^n n! (n-n_0)! e^{o(n)}$ trees, but of course a given tree may have been counted several times in this total number, so we need to estimate how often trees are counted.

Here, the bound on the number of leaves comes into play: given a specific tree, we bound the number of vertex orders it can arise from. Fixing some root vertex in $A_1$, the tree becomes a poset by the successor relation, which induces a partial order on those vertices that lie outside of $A_1$ (here, vertices that are not part of the tree are considered to be incomparable to all others). A tree can only come from a certain vertex order if that order is a linear extension of the partial order that was just described. It is well known (see \cite[Eq.~(5)]{Ruskey1992Generating} or \cite[Section 5.1.4, Ex.~20]{Knuth1998art}) that the number of linear extensions of a rooted forest $F$ is given by the formula
$$|F|! \prod_{v \in F} \frac{1}{s(v)},$$
where $s(v)$ is the number of successors of $v$, including $v$ itself, and the product is over all vertices. Note that $s(v) \geq 2$ as soon as $v$ is not a leaf. Thus each of the aforementioned trees, for which at least $\frac{n}{6} - o(n)$ of the vertices are not leaves, arises from at most
$$(n-n_0)! 2^{-n/6 + o(n)}$$
different vertex orders. So we can finally conclude that the number of trees of $G$ is, with high probability, at least
$$p^n n! (n-n_0)! e^{o(n)} \cdot \frac{2^{n/6 + o(n)}}{(n-n_0)!} = p^n n! e^{(n \log 2)/6 + o(n)},$$
which proves the lemma for any $a < (\log 2)/6$.
\end{proof}

\begin{lemma}\label{few_leaves}
For every $\alpha > 0$, there exists a $\beta > 0$ such that the number of trees with $n$ labelled vertices and fewer than $\beta n$ leaves is less than $n! e^{\alpha n}$ for sufficiently large $n$.
\end{lemma}

\begin{proof}
There is a classical bijective correspondence (going back to Cayley) between labelled trees with $n$ vertices and functions from $\{1,2,\ldots,n-2\}$ to $\{1,2,\ldots,n\}$, see~\cite[Chapter 1.7]{Harary1973Graphical}. In this bijection, the non-leaves correspond to the elements of the image. It follows that the number of trees with $n$ labelled vertices and exactly $\ell$ leaves is $\binom{n}{\ell} \cdot \Sur(n-2,n-\ell)$, where $\Sur(a,b)$ is the number of surjections from a set of $a$ elements to a set of $b$ elements. It is our goal to estimate the sum
$$\sum_{\ell < \beta n} \binom{n}{\ell} \Sur(n-2,n-\ell).$$
If $\beta \leq \frac12$, then we have
$$\sum_{\ell < \beta n} \binom{n}{\ell} \Sur(n-2,n-\ell) \leq \binom{n}{\lfloor \beta n \rfloor} \sum_{\ell < \beta n} \Sur(n-2,n-\ell)$$
by the unimodality of the binomial coefficients. Now consider arbitrary functions from the set $\{1,2,\ldots,n-2\}$ to itself. The number of functions for which $\{1,2,\ldots,n-L\}$ is contained in the image is greater than or equal to $\sum_{\ell = 2}^L \Sur(n-2,n-\ell)$, since clearly all surjections from $\{1,2,\ldots,n-2\}$ to some set of the form $\{1,2,\ldots,n-\ell\}$ (where $\ell \leq L$) have this property. Now let us bound the number of functions $f$ with this property: first, for every $j \in \{1,2,\ldots,n-L\}$, we pick an $i \in \{1,2,\ldots,n-2\}$ such that $f(i) = j$. This can be done in $(n-2)(n-3)\cdots (L-1)$ different ways. Now there are still $L-2$ elements in $\{1,2,\ldots,n-2\}$ left for which $f$ has no value yet. These values can be assigned in $(n-2)^{L-2}$ ways. Note that some functions are counted more than once in this way, but since we are only interested in upper bounds, this is immaterial. So we find that
$$\sum_{\ell=2}^{L} \Sur(n-2,n-\ell) \leq \frac{(n-2)!}{(L-2)!} (n-2)^{L-2}.$$
Thus
$$\sum_{\ell < \beta n} \binom{n}{\ell} \Sur(n-2,n-\ell) \leq \binom{n}{\lfloor \beta n \rfloor} \frac{(n-2)!}{(\lfloor \beta n \rfloor-2)!} (n-2)^{\lfloor \beta n \rfloor-2}.$$
By Stirling's formula,
$$\binom{n}{\lfloor \beta n \rfloor} \frac{(n-2)!}{(\lfloor \beta n \rfloor-2)!} (n-2)^{\lfloor \beta n \rfloor-2} = n! \exp \Big( \big( \beta - 2\beta \log \beta - (1-\beta)\log(1-\beta)\big) n + O(\log n) \Big).$$
Since $\beta - 2\beta \log \beta - (1-\beta)\log(1-\beta)$ tends to $0$ as $\beta \to 0^+$, we can make it smaller than $\alpha$ by choosing a sufficiently small $\beta$. The statement of the lemma then follows immediately.
\end{proof}

\begin{lemma}\label{lem:maxdeg_weakbound}
Suppose that $p \geq \frac{\log n}{2n}$. With high probability, the maximum degree of $G(n,p)$ is at most $4np$. 
\end{lemma}

\begin{proof}
We can apply the second Chernoff bound stated in Lemma~\ref{lem:chernoff} to the degree of a single vertex. It shows that the probability for the degree to be greater than $4(n-1)p$ is at most (taking $t = 3(n-1)p$) $e^{-9(n-1)p/4} = O(n^{-9/8})$ by our assumption on $p$. Thus the expected number of vertices whose degree is greater than $4(n-1)p$ is $O(n^{-1/8})$, and Markov's inequality shows that the maximum degree is at most $4np$ with high probability.
\end{proof}

We use this bound on the maximum degree to prove a modified version of the upper bound in Lemma~\ref{lem:Pk1}.

\begin{lemma}\label{lem:Pk_upper_hp}
Suppose that $p \geq \frac{\log n}{2n}$, and consider the random graph $G = G(n,p)$. There exists a constant $b > 0$ such that, with high probability, the inequality
$$\prs_k(G) \leq \st_{n-k}(G) b^k p^k n^k$$
holds for all $k \geq 0$.
\end{lemma}

\begin{proof}
We adapt the argument that we used in the proof of Lemma~\ref{lem:Pk1}. By Lemma~\ref{lem:maxdeg_weakbound}, we can assume that the maximum degree $\Delta(G)$ is at most $4np$. Given any subtree $S$ of $G$ with $n-k$ vertices, we estimate the number of ways to extend it to a spanning tree $T$. The remaining $k$ vertices induce at most $k \Delta(G)/2$ edges, so the number of ways to choose a forest of $r$ edges on this set of vertices is at most $\binom{k \Delta(G)/2}{r}$. As in the proof of Lemma~\ref{lem:Pk1}, the number of ways to extend $S$ and a forest of $r$ edges to a spanning tree $T$ is at most $\Delta(G)^{k-r} e^r$. Therefore, we find that
\begin{align*}
\prs_k(G) &\leq \st_{n-k}(G) \sum_{r=0}^{k-1} \binom{k \Delta(G)/2}{r} \Delta(G)^{k-r} e^r \leq \st_{n-k}(G)  \Delta(G)^k \Big( 1 + \frac{e}{\Delta(G)} \Big)^{k \Delta(G)/2} \\
& \leq \st_{n-k}(G)  \Delta(G)^k \Big( \exp \Big( \frac{e}{\Delta(G)} \Big) \Big)^{k \Delta(G)/2} = \st_{n-k}(G)  \Delta(G)^k  e^{ek/2}.
\end{align*}
Now we use the inequality $\Delta(G) \leq 4np$ to obtain
$$\prs_k(G) \leq \st_{n-k}(G) (4np)^k e^{ek/2} = \st_{n-k}(G) (4e^{e/2})^k p^k n^k.$$
In other words, the desired statement holds with $b = 4e^{e/2}$.
\end{proof}

\begin{proof}[Proof of Theorem~\ref{thm:dense_strong}]
We can assume that $p \geq \frac{\log n}{2n}$: if not, then the graph $G = G(n,p)$ is with high probability disconnected (in fact, it has an isolated vertex with high probability, see e.g.~\cite[Corollary 3.31]{Janson2000Random}), so that $\st_n(G) = \spp(G) = 0$. Thus Lemma~\ref{trees_lower_bound} applies. Let $a$ be the constant from Lemma~\ref{trees_lower_bound} for which
\begin{equation}\label{eq:tot_lower}
\tot(G) \geq p^n n! e^{an}
\end{equation}
holds with high probability, so that we can assume that this inequality is satisfied. Now pick a constant $\beta > 0$ according to Lemma~\ref{few_leaves} so that the number of trees with $n$ labelled vertices and fewer than $\beta n$ leaves is less than $n! e^{an/4}$ for sufficiently large $n$. Since each of these trees has probability $p^{n-1}$ of occurring as spanning tree in $G$, the expected number of spanning trees in $G$ with fewer than $\beta n$ leaves is less than $n! e^{an/4} p^{n-1}$. By the Markov inequality, the number of such spanning trees is therefore less than $p^n n! e^{an/2}$ with high probability. Hence we can assume that this is the case as well.

\medskip

Next, note that the inequality of Theorem~\ref{thm:dense_strong} holds (for sufficiently large $n$) if $\st_n(G) < p^n n! e^{3an/4}$ by~\eqref{eq:tot_lower} and our assumption on $p$. So we may also assume that $\st_n(G) \geq p^n n! e^{3an/4}$. This means, however, that at least $\st_n(G)(1-e^{-an/4})$ spanning trees of $G$ have at least $\beta n$ leaves each. In view of Lemma~\ref{lem:est_by_leaves}, this implies that
$$\prs_k(G) \geq \st_n(G)\big(1-e^{-an/4}\big) \binom{\lceil \beta n \rceil}{k}$$
for every $k$. On the other hand, making use of the inequality in Lemma~\ref{lem:Pk_upper_hp} (which is also satisfied with high probability), we obtain
$$\st_n(G)\big(1-e^{-an/4}\big) \binom{\lceil \beta n \rceil}{k} \leq \prs_k(G) \leq \st_{n-k}(G) b^k p^k n^k.$$
Thus
$$\frac{\st_{n-k}(G)}{\st_n(G)} \geq \big(1-e^{-an/4}\big) \binom{\lceil \beta n \rceil}{k}(bpn)^{-k}.$$
Summing over all values of $k$, we get
\begin{align*}
\frac{1}{\spp(G)} = \frac{\tot(G)}{\st_n(G)} &\geq \big(1-e^{-an/4}\big) \sum_{k \geq 0} \binom{\lceil \beta n \rceil}{k}(bpn)^{-k} = \big(1-e^{-an/4}\big) \Big( 1 + \frac{1}{b p n} \Big)^{\lceil \beta n \rceil} \\
&\geq \big(1-e^{-an/4}\big) \Big( 1 + \frac{1}{b p n} \Big)^{\beta n} =  \big(1-e^{-an/4}\big) e^{\beta/(b p) + o(1/p)}.
\end{align*}
This proves the desired inequality (with high probability, for sufficiently large $n$) for every $c < \beta/b$. 
\end{proof}

\section{Corollaries and final remarks}

Looking back over the proof of Theorem~\ref{thm:gnp} (the dense case), we see that we have also established the following as a side result:

\begin{cor}
Consider the random graph $G = G(n,p)$, and suppose that $p \to p_{\infty} > 0$. As $n \to \infty$, we have, for every fixed nonnegative integer $k$,
$$\frac{\st_{n-k}(G)}{\st_n(G)} \overset{p}{\to} \frac{1}{k!} (ep_{\infty})^{-k}.$$
\end{cor}

In other words, the number of vertices missing in a random subtree asymptotically follows a Poisson distribution.

In fact, the same proof still works if $k$ and $\frac{1}{p}$ both grow sufficiently slowly with $n$, in which case we have
$$k! (ep_{\infty})^k \frac{\st_{n-k}(G)}{\st_n(G)} \overset{p}{\to} 1.$$
It might be interesting to establish the precise thresholds for $k$ and $p$ up to which this statement remains true.

The average number of edges in a random subtree is another parameter that was considered by Chin et al.~in \cite{Chin2018Subtrees}. Equivalently, one can also consider the average number of vertices, which is precisely $1$ greater. For this quantity, we obtain the following result in the dense case:

\begin{cor}
Let $\mu(G)$ denote the average number of edges in a randomly chosen subtree of $G$. Consider the random graph $G = G(n,p)$, and suppose that $p \to p_{\infty} > 0$. As $n \to \infty$, we have
$$n - \mu(G) \to 1 + \frac{1}{ep_{\infty}}.$$
\end{cor}

Moreover, we can combine Janson's Theorem~\ref{thm:janson}, Theorem~\ref{thm:dense} and Slutsky's theorem \cite[Chapter 5, Theorem 11.4]{Gut2013Probability} to obtain the following result on the distribution of the total number of subtrees:

\begin{cor}
Assume that $p \to p_{\infty} \in (0,1)$. Then we have, as $n \to \infty$,
$$\log \tot(G(n,p)) - \log (n^{n-2}p^{n-1}) \overset{d}{\to} N \Big( \frac{1-e+ep_{\infty}}{ep_{\infty}}, \frac{2-2p_{\infty}}{p_{\infty}} \Big).$$
\end{cor}

The argument fails if $p \to 0$ since the difference between $\log \tot(G(n,p))$ and $\log \st_n(G(n,p))$ grows at a rate of $\frac{1}{p}$, while the standard deviation of $\log \st_n(G(n,p))$ is only of order $\frac{1}{\sqrt{p}}$. However, it should be possible to adapt Janson's proof of Theorem~\ref{thm:janson} to obtain a lognormal limit law for $\tot(G(n,p))$ as long as $\liminf_{n \to \infty} \sqrt{n}p > 0$.

\section*{Acknowledgment}

The author would like to thank Svante Janson for many helpful comments.

\end{document}